\documentclass[runningheads]{llncs}
\usepackage{graphicx} 
\usepackage{algorithm,algorithmicx,algpseudocode}

\usepackage{theorem}
\usepackage{amsmath, amssymb, mathtools}
\usepackage{graphics}
\usepackage{enumerate}
\usepackage{comment}
\usepackage{bm}
\usepackage{breqn}
\usepackage{cases}
\usepackage{eucal}
\usepackage{fancyvrb}
\usepackage{seqsplit}
\usepackage{autobreak}
\usepackage{color}
\usepackage{pb-diagram}
\makeatletter
\def\sbar{\accentset{{\cc@style\underline{\mskip16mu}}}}
\def\mbar{\accentset{{\cc@style\underline{\mskip32mu}}}}
\def\lbar{\accentset{{\cc@style\underline{\mskip48mu}}}}

\makeatother

\newcommand{\bP}{\mathbb{P}}

\newtheorem{Def}{Definition}[section]

\newtheorem{Prop}{Proposition}
\newtheorem{Lem}[Def]{Lemma}
\newtheorem{Cor}[Def]{Corollary}
\newtheorem{Te}[Def]{Definition}

\newtheorem{MMain}{Theorem}
\newtheorem{Conj}{Conjecture}
\newcommand{\C}{{\mathbb Z}}

\begin{document}

\title{Automorphism groups of hyperelliptic curves\\ of $2$-rank zero}
\author{Kohtaro Yamaguchi\inst{1} and Shushi Harashita\inst{2}
}

\authorrunning{K. Yamaguchi and S. Harashita}
\institute{
Graduate School of Environment and Information Sciences, Yokohama National University
79-7 Tokiwadai, Hodogaya-ku, Yokohama 240-8501 Japan
\email{yamaguchi-kohtaro-vx@ynu.jp} \and
Graduate School of Environment and Information Sciences, Yokohama National University
79-7 Tokiwadai, Hodogaya-ku, Yokohama 240-8501 Japan
\email{harasita@ynu.ac.jp}
}
\maketitle

\begin{abstract}
\noindent\quad 
In this paper, we determine the reduced automorphism groups of hyperelliptic curves of a small genus in characteristic $2$, when they are of $2$-rank $0$. Such a curve is an Artin-Schreier curve defined in the form $y^2-y=f(x)$ for a polynomial $f(x)$. After we clarify semidirect-product structures of the automorphism groups for an arbitrary genus,
we derive the detailed group structures for the reduced automorphism groups of the curves of a small genus, through computations using the computational algebra system Magma. With these experiments, we formulate two conjectures, which are analogues for our curves of the Oort conjecture on automorphism groups of generic principally polarized supersingular abelian varieties. 
\end{abstract}

\section{Introduction}
In this paper, we investigate the order and the group structure of automorphism group of hyperelliptic curves in characteristic $2$ of small genus which
is of $2$-rank 0
or is supersingular. 
This research is motivated in part by Oort's conjecture, which predicted that the automorphism group of generic supersingular principally polarized abelian variety of dimension $g\ge 2$ is $\{\pm 1\}\simeq \C_2$, where $\C_m$ denotes the cyclic group of order $m$.
As an analogue of this conjecture, one may ask whether
any generic $2$-rank $0$ (resp. supersingular) hyperelliptic curve in characteristic $2$ of genus $\ge 2$ is  isomorphic to $\C_2$.  Our theorems 
answers this question for small genera, which lead us to formulate refined conjectures.
Although Oort's conjecture has been resolved affirmatively recently (\cite{KYY} for $g=3$, \cite{KY} for even $g$ and \cite{Viehmann} for general $g$) unless $g=2,3$ and $p=2$,
the purpose of this paper is to highlight the interest of studying the case of a class of principally polarized abelian varieties that is, in one sense, narrower and, in another sense, broader than original Oort's case.

The classification of the autormorphism groups of curves contributes to the classification of curves. The representative prior studies are as follows.
In \cite[Section 8]{Igusa}, Igusa classified the automorphism groups of curves of genus two into seven types.
The case of genus-three hyperelliptic curves was done by Lercier and Ritzenthaler \cite[3.1]{LR}. The aim of this paper can be regarded as obtaining similar results when we restrict ourselves to the case of hyperelliptic curves in characteristic $p=2$ which is of $p$-rank zero or supersingular. Here the $p$-rank of a curve $C$ over an algebraically closed field $k$ in characteristic $p$ is the logarithm with base $p$ of the cardinality of $k$-valued points of the $p$-kernel of the Jacobian variety of $C$. We say that a curve $C$ is supersingular if its Jacobian variety is supersingular, i.e., is isogenous to a product of supersingular elliptic curves.
The $p$-rank of a supersingular curve is known to be zero, but the fact that $C$ has $p$-rank $0$ does not imply that $C$ is supersingular if the genus of $C$ is greater than or equal to $3$.


Let $K$ be an algebraically closed field of characteristic $2$.
We study the following Artin-Schreier curve
\[
    C:y^2-y=f(x),
\]
where
\[
f(x):=x(x^{2n}+a_{n-1}x^{2(n-1)}+\cdots +a_1x^2+a_0)
\]
with $a_i\in K$.
As we review in Section \ref{sec:hyperelliptic}, it is known that such curves $C$ are of $2$-rank $0$ and conversely any hyperelliptic curve
of $2$-rank $0$ over $K$ is realized in this way. Consider a coordinate-change of the form
\begin{equation}\label{eq:CoordinateTransform}
\left\{
\begin{array}{l}
x \mapsto X:=\alpha x+\beta, \\
y \mapsto Y:=y+\gamma_nx^n+\gamma_{n-1}x^{n-1}+\cdots +\gamma_1x+\gamma_0
\end{array}
\right.
\end{equation}
for $\alpha, \beta, \gamma_0,\ldots,\gamma_n\in K$ with $\alpha\ne 0$.
If $Y^2-Y=f(X)$ holds, then this coordinate-change defines an automorphism of $C$. Conversely, we shall see in Proposition \ref{prop:auto} that
any automorphism of $C$ is realized in this way.
We denote this automorphism by 
\[
(\alpha, \beta; \gamma_n, \ldots, \gamma_0).
\]
Let ${\rm Aut}(C)$ be the automorphism group of $C$.
The automorphism $\sigma$ defined by
\[\sigma:
\left\{
\begin{array}{l}
x \mapsto x, \\
y \mapsto y+1,
\end{array}
\right.
\]
i.e., $\sigma = (1, 0; 0, \ldots, 0, 1)$ is the hyperelliptic involution, since $C/\langle \sigma \rangle$ is isomorphic to ${\mathbb P}^1$.
The group ${\rm Aut}(C)/\langle \sigma \rangle$ is called {\it the reduced automorphism group}. We write it as ${\rm RA}(C)$.

For an automorphism $(\alpha, \beta; \gamma_n, \ldots, \gamma_0)$ of $C$, we obtain $\alpha^{2n+1}=1$ by examining the leading term of $Y^2 - Y = f(X)$. We have the following two homomorphisms
\[
\rho_{n}:\quad {\rm Aut}(C) \longrightarrow \mu_{2n+1} 
\]
sending $(\alpha,\beta;\gamma_n,\dots,\gamma_0)$ to $\alpha$, where $\mu_{2n+1}$ is the group of
$(2n+1)$-th roots of unity, and
\[
\tau_{n}:\quad {\rm Aut}(C) \longrightarrow M_{n}
\]
sending $(\alpha,\beta;\gamma_n,\dots,\gamma_0)$ to $(\alpha,\beta)$,
where 
$M_n$ is the group of affine transformations $(u,v): x\mapsto u x + v$
for $u,v\in K$.

The main results of this paper are Theorems \ref{thm:1}, \ref{thm:2}, and \ref{thm:3} below. The first theorem reveals the group extension structure of 
${\rm Aut}(C)$.
\begin{MMain}\label{thm:1}
\begin{enumerate}
\item[\em (1)] We have ${\rm Ker}\tau_n = \langle \sigma \rangle$ and therefore 
\[
{\rm RA}(C):={\rm Aut}(C)/\langle \sigma \rangle \simeq {\rm Im}\tau_n. 
\]
\item[\em (2)] Put $U_n:=\left\{
\left. \begin{pmatrix}
    1 & \beta \\
    0 & 1
\end{pmatrix}
\ \right| \ (1, \beta)\in {\rm Im}\tau_n
\right\}$. Then we have
\[
{\rm Im}\tau_n \simeq \left\{
\left. \begin{pmatrix}
    \alpha & \beta \\
    0 & 1
\end{pmatrix}
\ \right| \  (\alpha, \beta)\in {\rm Im}\tau_n
\right\}
\simeq U_n\rtimes \C_{\#{\rm Im}\rho_n}. 
\]
Moreover, since $U_n$ is an elementary abelian $2$-group, there exists a non-negative integer $\ell$ such that $U_n\simeq \C_2^{\ell}$;
hence
\[
{\rm RA}(C)\simeq \C_2^{\ell}\rtimes \C_{\#{\rm Im}\rho_n}. 
\]
\item[\em (3)] If $n$ is odd with $n\geq 3$, then we have ${\rm Ker}\rho_n = \langle \sigma \rangle$, equivalently $\# U_n = 1$.
\item[\em (4)] We have
\[
{\rm Aut}(C) \simeq {\rm Ker}\rho_n\rtimes \C_{\#{\rm Im}\rho_n}.
\]
In particular, by (3), if $n$ is odd with $n\geq 3$, then we have
${\rm Aut}(C)\simeq \C_2\rtimes \C_{\#{\rm Im}\rho_n}$.
\end{enumerate}
\end{MMain}
As a result, we found that these groups are isomorphic to groups constructed as direct products and semidirect products of several cyclic groups. In the subsequent, Theorem \ref{thm:2}, we compute, using the computer algebra system Magma, a Gr\"obner basis of the ideal generated by the conditions under which the coordinate transformation \eqref{eq:CoordinateTransform} becomes an automorphism. For $n=1,\ldots ,6$, this yields the detailed group structure and defining conditions of the reduced automorphism group $\mathrm{RA}(C)$, as well as the order of the automorphism group $\mathrm{Aut}(C)$ (cf. \cite{Yamaguchi}). The results are as follows.

\begin{MMain}\label{thm:2}
For $n=1,\ldots, 6$, we have the table of
the group structure of ${\rm RA}(C)$
and the order of ${\rm Aut}(C)$.
\begin{center}
\begin{tabular}{|c|l|c|}
\hline
$n$ & ${\rm group\ structure\ of\  RA}(C)$ & $\#{\rm Aut}(C)$ \\
\hline
$1$ & $\C_2^2\rtimes \C_3\simeq A_4$ & $24$ \\
\hline
$2$ & $\C_2^4 \ {\rm if\ }a_1\neq0$ & $32$ \\
\cline{2-3}
& $\C_2^4\rtimes \C_5$\ \ \rm{if} $a_1=0$ & $160$ \\
\hline
$3$ & $\C_1\ \ {\rm if\ }a_2^8+a_2+a_1^4\neq0\ {\rm or}\ a_2^{12}+a_2a_1^2+a_0^4\neq0$ & $2$ \\
\cline{2-3}
    & $\C_7 \ \  {\rm if\ }a_2^8+a_2+a_1^4=a_2^{12}+a_2a_1^2+a_0^4=0$ & $14$ \\
\hline
$4$ & 
$\C_1$\ \ {\rm if\ }$a_3\ne0$ & $2$ \\
\cline{2-3}
&$\C_2^6 \ \  {\rm if\ }a_3=0\ {\rm and }\ a_2\ne0$ & $128$ \\
\cline{2-3}
&$\C_2^6\rtimes \C_3 \ \  {\rm if\ }a_3=a_2=0\ {\rm and}\ a_1\neq0$ & $384$ \\
\cline{2-3}
&$\C_2^6\rtimes \C_9 \quad {\rm if\ }a_3=a_2=a_1=0$ & $1152$ \\
\hline
$5$ & $\C_1$\ \ 
\begin{tabular}{l}
{\rm if\ } $a_3\neq0 $\ {\rm or}\ $a_4+a_2^4\neq0$\ {\rm or}\ $a_4^8+a_1^2\neq0$\\
\phantom{aa }{\rm or}\ $a_4^{40}+a_4^{18}+a_4a_2^2+a_0^8\neq0$ 
\end{tabular} 
 & $2$\\
\cline{2-3}
&$\C_{11}\ \ {\rm if\ }a_3=a_4+a_2^4=a_4^8+a_1^2=a_4^{40}+a_4^{18}+a_4a_2^2+a_0^8=0$ & $22$ \\
\hline
$6$ & $\C_1$\ \ 
\begin{tabular}{l}
         {\rm if\ either}\ $a_5(a_3^2+a_5^6)=0$,\ $a_5a_3^4+a_5^{13}+1\neq0$,\ {\rm or}\ $t\neq0$ \\
\phantom{a }{\rm is true, and at\ least\ of}
         \\
         \phantom{a}$\ a_5,\ a_3,\ a_4^4+a_2^2,\ a_4+a_1^{16},\ a_4^{24}+a_4^4a_1^8+a_4a_1^4+a_0^8$\ {\rm is\ not\ 0}\end{tabular} & $2$ \\
\cline{2-3}
& $\C_2\ \ \begin{tabular}{l}${\rm if\ }a_5(a_3^2+a_5^6)\neq0\ {\rm and}\ a_5a_3^4+a_5^{13}+1=t=0, $ {\rm and at\ least\ of}\\
    \quad $\ a_5,\ a_3,\ a_4^4+a_2^2,\ a_4+a_1^{16},\ a_4^{24}+a_4^4a_1^8+a_4a_1^4+a_0^8\ {\rm is\ not\ 0}$\end{tabular}$ & $4$ \\
\cline{2-3}
& $\C_{13}\ \ {\rm if\ }a_5=a_3=a_4^4+a_2^2=a_4+a_1^{16}=a_4^{24}+a_4^4a_1^8+a_4a_1^4+a_0^8=0$ & $26$ \\
\hline
\end{tabular}
\end{center}
Here
\begin{eqnarray}\label{eq:20260412}
t&=&a_1^8+a_1^4a_3^{24}+a_1^4a_5^{72}+a_1^4a_5^{46}+a_1^4a_5^{20}+a_2^8a_5^8+a_2^2a_3^{28}+a_2^2a_5^{84}+a_2^2a_5^{71}\nonumber\\&&
+a_2^2a_5^{58}+a_2^2a_5^{45}+a_2^2a_5^{32}+a_2^2a_5^{19}+a_2^2a_5^6+a_3^{30}a_4+a_3^{22}+a_3^4a_4+a_3^2a_4a_5^{84}\nonumber\\&&
+a_3^2a_4a_5^{71}+a_3^2a_4a_5^{58}+a_3^2a_4a_5^{45}+a_3^2a_4a_5^{32} +a_3^2a_4a_5^{19}+a_3^2a_4a_5^6+a_3^2a_5^{60}\nonumber\\&&
+a_3^2a_5^{47}+a_3^2a_5^8+a_4^8a_5^{24}+a_4a_5^{12}+a_5^{40}+a_5^{14}. 
\end{eqnarray}

\end{MMain}
The case of genus $1$ is well-known.
Since a genus-2 curve is supersingular if and only if it is of $p$-rank $0$ (cf. \cite[Lem.\ 1.1 (i)]{IKO}),
the result for genus $2$ gives a counter-example of original Oort conjecture, but this has already been proved by Ibukiyama \cite{Ibukiyama} in a different way. 
In the theorem, we have compiled these cases into a single table so that the reader can compare them.

By this theorem,  generic hyperelliptic curves of $2$-rank $0$ (in characteristic $2$) is $\{\pm 1\}$ have the automorphism group $\{\pm1\}$ (an analogue of Oort conjecture is true) for
$2<g\leq 6$. It would be natural to expect

\begin{Conj}
The automorphism group of a generic hyperelliptic curve of genus $g$ and $2$-rank $0$ in characteristic $2$ is $\{\pm 1\}$ if $g > 2$.
\end{Conj}

Finally we study supersingular curves
in characteristic $2$ of  genus $\leq 9$. Scholten-Zhu \cite{Scholten} gave
a defining equation of such a curve as in the table below.
We call them Artin-Schreier curves of Scholten-Zhu type.
\begin{center}
\begin{tabular}{|c|l|}
\hline
{\rm genus} $g$ & {\rm Artin-Schreier curves of Scholten-Zhu type} $C_g^{SZ}$ \\
\hline
$1$ & $y^2-y=x^3$ \\
$2$ & $y^2-y=x^5+c_3x^3$ \\
$3$ & {\rm none} \\
$4$ & $y^2-y=x^9+c_5x^5+c_3x^3$ \\
$5$ & $y^2-y=x^{11}+c_3x^3+c_1x$ \\
$6$ & $y^2-y=x^{13}+c_3x^3+c_1x$ \\
$7$ & {\rm none} \\
$8$ & $y^2-y=x^{17}+c_9x^9+c_5x^5+c_3x^3$ \\
$9$ & $y^2-y=x^{19}+c^8x^9+c^3x$ \\
\hline
\end{tabular}
\end{center}
Remark that there is no supersingular hyperelliptic curve in characteristic $2$ if the genus is of the form $2^k-1$
for some integer $k\ge 2$, see \cite[Thm.\ 1.2]{SZ:hyperell}.

We computed the automorphism groups of Artin-Schreier curves of Scholten-Zhu type:
\begin{MMain}\label{thm:3}
The group structure of the reduced automorphism group ${\rm RA}(C_g^{SZ})$ and the order of the automorphism group ${\rm Aut}(C_g^{SZ})$, under the coordinate transformation $(\ast)$ of the curve $C_g^{SZ}$ are as follows. 
\begin{center}
\begin{tabular}{|c|l|c|}
\hline
{\rm genus} $g$ & {\rm group structure of} ${\rm RA}(C^{SZ}_{g})$ & $\#{\rm Aut}(C^{SZ}_g)$ \\
\hline
$1$ & $\C_2^2\rtimes \C_3\simeq A_4$ & $24$\\
\hline
$2$ & $\C_2^4 \quad {\rm if\ }c_3\neq0$ & $32$\\
\cline{2-3}
& $\C_2^4\rtimes \C_5 \quad {\rm if\ }c_3=0$ & $160$ \\
\hline
$4$ & $\C_2^6 \quad {\rm if\ }c_5\neq0$ & $128$\\
\cline{2-3}
& $\C_2^6\rtimes \C_3 \quad {\rm if\ }c_5=0\ {\rm and}\ c_3\neq0$ & $384$\\
\cline{2-3}
& $\C_2^6\rtimes \C_9 \quad {\rm if\ }c_5=c_3=0$ & $1152$ \\
\hline
$5$ & $\C_1 \quad {\rm if\ }c_3\neq0\ {\rm or}\ c_1\neq0$ & $2$\\
\cline{2-3}
& $\C_{11} \quad {\rm if\ }c_3=c_1=0$ & $22$ \\
\hline
$6$ & $\C_1 \quad {\rm if\ }c_3\neq0\ {\rm or}\ c_1\neq0$ & $2$\\
\cline{2-3}
& $\C_{13} \quad {\rm if\ }c_3=c_1=0$ & $26$ \\
\hline
$8$ & $\C_2^8 \quad {\rm if\ }c_9\neq0\ {\rm or}\ c_5\neq0\ {\rm or}\ c_3\neq0$ & $512$\\
\cline{2-3}
& $\C_2^8\rtimes \C_{17} \quad {\rm if\ }c_9=c_5=c_3=0$ & $8704$ \\
\hline
$9$ & $\C_1 \quad {\rm if\ }c\neq0$ & $2$\\
\cline{2-3}
& $\C_{19} \quad {\rm if\ }c=0$ & $38$ \\
\hline
\end{tabular}
\end{center}
\end{MMain}

A naive hyperelliptic analogue of Oort conjecture would predict that
any generic supersingular hyperelliptic curve of genus $\ge 2$ has the automorphism group $\{\pm 1\}$. However
from this theorem this is false in characteristic $2$ if $g$ is $2,4,8$. 
A reasonable refinement of the above conjecture would be the following:
\begin{Conj}
The automorphism group of a generic supersingular hyperelliptic curve of genus $g$ in characteristic $2$ is $\{\pm 1\}$ if $g$ is not a power of $2$. More strongly one may expect that this would hold even if we replace ``if" by ``if and only if".
\end{Conj}

This paper is organized as follows.
In Section 2, we review a description of hyperelliptic curves in characteristic $2$ which are of $2$-rank $0$, and a basic fact on their automorphism groups.
In Section 3, we prove the three  theorems in Introduction.
In Section 4, we present a summary of this paper and discuss the issues to be addressed in future work.

\subsection*{Acknowledgments}
The first author would like to express his deepest gratitude to Professor Harashita, his supervisor, for the generous guidance and valuable advice provided throughout the preparation of this paper. The authors are also grateful to the members of the Harashita Laboratory for their insightful comments during seminars and discussions. They also thank the referees for their comments and remarks.

\section{Preliminaries}
In this section, we review the basics of hyperelliptic curves, which are of $2$-rank $0$, and their automorphism groups.

\subsection{Hyperelliptic curves of $2$-rank zero}\label{sec:hyperelliptic}
Let $K$ be an algebraically closed field of characteristic $2$.
Let $C$ be a hyperelliptic curve over $K$.
The function field of $C$ is a quadratic extension of the field $K(x)$ of rational functions, say $K(x)[y]$ with $y^2+a(x)y=b(x)$ for $a(x), b(x) \in K(x)^\times$.
By replacing $y$ by $a(x)y$, we see that $C$ is isomorphic to an Artin-Schreier curve
\[
y^2 - y = f(x)
\]
for $f(x) \in K(x)$. Assume that the $2$-rank of $C$ is zero.
Then it follows from \cite[Lem.\ 2.6]{Pries} that $f(x)$ has a single pole on the projective line ${\mathbb P}^1$. By a coordinate change, one may assume that the pole is at the point $P_\infty$ of infinity, in which case $f(x) \in K[x]$. By replacing $y$ by $y+h$ for some polynomial $h\in K[x]$, we may assume that $f(x)$ does not contain any even-degree terms.
Hence we conclude that every hyperelliptic curve of $2$-rank $0$ is realized as
\begin{equation}\label{eq:OurCurve}
y^2-y=x(x^{2n}+a_{n-1}x^{2(n-1)}+\cdots +a_1x^2+a_0)
\end{equation}
for $a_i\in K$. Conversely any nonsingular curve defined as above
is hyperelliptic curves of genus $n$ and $2$-rank zero (cf.\ \cite[Lem.\ 2.6]{Pries}).

\subsection{Automorphisms}
Let $C$ be a hyperelliptic curves of $2$-rank 0, say defined as in \eqref{eq:OurCurve}.

\begin{Prop}\label{prop:auto}
The automorphisms of the curve $C$ are only those induced by coordinate transformations \eqref{eq:CoordinateTransform}. 
\end{Prop}

\begin{proof}
Since 
the morphism $\pi:C\rightarrow \bP^1$ of degree two 
is unique up to isomorphism by \cite[Chap.\ IV, Prop.\ 5.3]{Hartshorne},
for an automorphism $\psi$ of $C$
there exists an automorphism $\varphi$ of $\bP^1$
that makes the diagram commute
\[
  \begin{diagram}
    \node{C} \arrow{e,t}{\psi} \arrow{s,l}{\pi} \node{C} \arrow{s,r}{\pi} \\
    \node{\bP^1} \arrow{e,t}{\varphi} \node{\bP^1.}
  \end{diagram}
\]
Since the point $P_\infty$ at infinity of ${\mathbb P}^1$ is the unique branch point of $\pi$, the point $P_\infty$ is not moved by $\varphi$.
Hence $\psi$ induces an automorphism of $C\backslash \pi^{-1}(P_\infty)$.
Let $x$ and $X$ be the usual $x$-coordinate of ${\mathbb P}^1$
of the source and the target of $\varphi$.
Since $\varphi$ stabilizes $P_\infty$, 
$\varphi$ sends $x$ to $\alpha X+\beta$.
Write
$f(x):=x(x^{2n}+a_{n-1}x^{2(n-1)}+\cdots +a_1x^2+a_0)$.
The coordinate ring of $C\backslash \pi^{-1}(P_\infty)$ is
\[
K[x,y]/(y^2-y-f(x)).
\]
The associated isomorphism $\psi^*$ of coordinate rings

\[
\psi^*:K[x,y]/(y^2-y-f(x)) \longrightarrow K[X,Y]/(Y^2-Y-f(X))
\]
sends $x$ to $\alpha X+\beta$ and $y$ to $h_1 Y+h_2 $ for some $h_1, h_2\in K[X],\ h_1\neq0$.
It sends $y^2-y-f(x)$ to
\[
(h_1 Y+h_2 )^2-(h_1 Y+h_2 )-f(\alpha X+\beta) =
(h_1^2-h_1)Y+h_1^2f(X) -f(\alpha X+\beta)+ h_2^2-h_2
\]
in $K[X,Y]/(Y^2-Y-f(X))$, which is zero.
If $h_1^2=h_1$ did not hold, then $Y$ would be written only by $X$, which contradicts that $\pi$ is of degree $2$. Hence we have $h_1=1$ as $h_1 \neq 0$. It is clear that the degree of $h_2$ has to be less than or equal to $n$. 
\end{proof}

\section{Proof}
In this section, we prove the theorems presented in Section 1. 
The notation introduced in Section 1 will be used freely throughout this section.

\subsection{Proof of Theorem \ref{thm:1}}
Recall that we study a nonsingular curve of the form
\[
    C:y^2-y=f(x):=x(x^{2n}+a_{n-1}x^{2(n-1)}+\cdots +a_1x^2+a_0)
\]
for $a_i\in K$ and that an element of ${\rm Aut}(C)$ is of the form
\[
(\alpha, \beta; \gamma_n, \ldots, \gamma_0)
\]
for some $\alpha, \beta, \gamma_0,\ldots, \gamma_n\in K$
which denotes a coordinate-change
\begin{equation}\label{eq:Sec3-CoordinateTransform}
\left\{
\begin{array}{l}
x \mapsto X:=\alpha x+\beta, \\
y \mapsto Y:=y+\gamma_nx^n+\gamma_{n-1}x^{n-1}+\cdots +\gamma_1x+\gamma_0
\end{array}
\right.
\end{equation}
and satisfies $Y^2-Y=f(X)$.
We now begin the proof of Theorem \ref{thm:1}.
\begin{proof}
\begin{enumerate}
\item[(1)]
What we want to show is that
\[
\mbox{Ker}\tau_n=\{(\alpha, \beta; \gamma_n, \ldots, \gamma_0)\in \mbox{Aut}(C)\mid (\alpha, \beta)=(1,0)\}\]
is generated by the hyperelliptic involution $\sigma$.
Recall $\sigma=(1, 0; 0, \ldots, 0, 1)$. 
We see that $(\alpha, \beta)=(1,0)$ implies $\gamma_n=\cdots =\gamma_1=0$ and $\gamma_0=0, 1$, by examining the $x^{2i}$-coefficient  of $Y^2-Y=f(X)$
 with $X$ and $Y$ as in \eqref{eq:Sec3-CoordinateTransform} for $i=0,\ldots,n$.
 Hence we have ${\rm Ker}\tau_n= \langle \sigma \rangle \simeq \C_2$ and therefore 
${\rm RA}(C) = {\rm Aut}(C)/{\rm Ker}\tau_n \simeq {\rm Im}\tau_n.$

\item[(2)]
For two elements $(\alpha, \beta), (A, B)\in {\rm Im}\tau_n$, their product in ${\rm Im}\tau_n$ is given by 
\[
(\alpha, \beta)\cdot (A, B)=(\alpha A, \alpha B +\beta). 
\]
The first isomorphism ${\rm Im}\tau_n \simeq \left\{
\left. \begin{pmatrix}
    \alpha & \beta \\
    0 & 1
\end{pmatrix}
\ \right| \  (\alpha, \beta)\in {\rm Im}\tau_n
\right\}$ is obtained by comparing with the product of the associated matrices. To show the second isomorphism, it suffices to show that $U_n:=\left\{
\left. \begin{pmatrix}
    1 & \beta\\
    0 & 1
\end{pmatrix}
\ \right| \ (1, \beta)\in {\rm Im}\tau_n
\right\}$ is a normal subgroup of ${\rm Im}\tau_n$, that a complement of $U_n$ in ${\rm Im}\tau_n$ is isomorphic to $\C_{\#{\rm Im}\rho_n}$, and that a complement of $U_n$ exists in ${\rm Im}\tau_n$. We consider the group homomorphism
 \[
{\rm Im}\tau_n \to \mu_{2n+1}
\]
sending $\displaystyle \begin{pmatrix}
            \alpha & \beta \\
            0 & 1
        \end{pmatrix}$ to $\alpha$. Since the kernel of this homomorphism is $U_n$, we have $U_n\triangleleft {\rm Im}\tau_n$. Also, since the image of this homomorphism is a cyclic group of order $\#{\rm Im}\rho_n$, by the fundamental theorem on homomorphisms we have ${\rm Im}\tau_n/U_n\simeq \C_{\#{\rm Im}\rho_n}$. Here, the complement of $U_n$ in ${\rm Im}\tau_n$ is isomorphic to ${\rm Im}\tau_n/U_n$. Finally, we show that $U_n$ has a complement in ${\rm Im}\tau_n$. From the Schur-Zassenhaus Lemma (cf. \cite[Chap.7, Thm.7.41]{Rotman}), it is sufficient to show that the orders of $U_n$ and ${\rm Im}\tau_n/U_n$ are coprime. For any $(1, \beta)\in {\rm Im}\tau_n$, since $(1, \beta)\cdot (1, \beta)=(1, 0)=e$ holds, the order of any nontrivial element of $U_n$ is $2$. Therefore, $U_n$ is either the trivial group $\{e\}$ or an elementary abelian $2$-group, and its order is found to be a power of $2$. On the other hand, since the order of ${\rm Im}\tau_n/U_n\simeq \C_{\#{\rm Im}\rho_n}$ is a divisor of $2n+1$, that is, an odd number, the orders of $U_n$ and its complement in ${\rm Im}\tau_n$ are coprime. From the above, the second isomorphism was shown. From the proof so far and Theorem \ref{thm:1} (1), there exists a non-negative integer $\ell$  such that ${\rm RA}(C)\simeq \C_2^{\ell}\rtimes \C_{\#{\rm Im}\rho_n}$. 
        
\item[(3)] We determine
\[
{\rm Ker} \rho_n =\{(\alpha, \beta; \gamma_n, \ldots, \gamma_0)\mid \alpha=1\},
\]
when $n$ is odd with $n\ge 3$. The $x^\ell$-coefficient $c(\ell)$ 
of the right hand side of the equation $Y^2-Y=\sum_{\ell = 0}^{2n+1}c(\ell)X^\ell$ obtained by applying the coordinate transformation \eqref{eq:Sec3-CoordinateTransform} to the curve $C$ is 
\[
c(\ell)=\binom{2n+1}{\ell}\alpha^\ell\beta^{2n+1-\ell}+\sum_{i=\lfloor \frac{\ell}{2} \rfloor}^{n-1}a_i\binom{2i+1}{\ell}\alpha^\ell\beta^{2i+1-\ell}+\gamma_{\frac{\ell}{2}}^2+\gamma_\ell,
\]
where we put $\gamma_{\ell} = 0$ for $\ell>n$ and $\gamma_{\frac{\ell}{2}} = 0$ if $\frac{\ell}{2}\notin\mathbb{N}$. The following equation is a necessary and sufficient condition for the coordinate transformation \eqref{eq:Sec3-CoordinateTransform} to be an automorphism. 
\begin{equation}\label{eq:ConditionAut}
    c(\ell)=
        \begin{cases}
        1 & (\ell=2n+1) \\
        0 & (\ell:{\rm even}) \\
        a_{\frac{\ell-1}{2}} & (\ell:\text{odd and } \ell \ne 2n+1)
        \end{cases}
\end{equation}
Let $n\ge 3$ be an odd integer. To consider the kernel of $\rho_n$, we assume $\alpha=1$. 
Since
\[
c(2n-1) = \binom{2n+1}{2n-1}\beta^2+a_{n-1}\binom{2n-1}{2n-1}\beta^0 = \beta^2+a_{n-1},
\]
the equation \eqref{eq:ConditionAut} reads $\beta^2+a_{n-1}=a_{n-1}$. Hence we obtain $\beta=0$. Also, we have $\gamma_n=0$ and $\gamma_0=0, 1$ from $c(2n)=\beta+\gamma_n^2=0$ and $c(0)=\gamma_0^2+\gamma_0=0$. 
Let $m$ be an integer with $1\leq m\leq n$.
By $\beta=0$, we have
\[
\begin{split}
c(2m-1)&=\binom{2n+1}{2m-1}\beta^{2(n-m+1)}+\sum_{i=m-1}^{n-1}a_i\binom{2i+1}{2m-1}\beta^{2(i-m+1)}+\gamma_{2m-1}\\
&=a_{m-1}+\gamma_{2m-1}.
\end{split}
\]
The equation \eqref{eq:ConditionAut} reads $a_{m-1}+\gamma_{2m-1}=a_{m-1}$ and therefore
we have $\gamma_{2m-1}=0$. 
Finally, we show that $\gamma_{\ell}=0$ holds for any even number $\ell\ge2$. From $\beta=0$, when $\ell\ge2$ is even, 
\[
c(\ell)=\gamma_{\frac{\ell}{2}}^2+\gamma_{\ell}=0
\]
holds. By substituting $\ell=2$, we see that $\gamma_2=0$, and for even $\ell$ such that $4\le l \le 2n$, it follows inductively that $\gamma_{\ell}=0$. From the above, we have 
\[
{\rm Ker}\rho_{n}
=\{(1,0;0,\ldots,0,0),(1,0;0,\ldots,0,1) \}
= \langle \sigma \rangle. 
\]

\item[(4)]
Since ${\rm Aut}(C)/{\rm Ker}\rho_n\simeq {\rm Im}\rho_n$, 
if $\#{\rm Ker}\rho_n$ and $\#{\rm Im}\rho_n$ are coprime, then ${\rm Aut}(C)\simeq {\rm Ker}\rho_n\rtimes {\rm Im}\rho_n$ that we aim to prove here
follows from the Schur-Zassenhaus Lemma.
Since ${\rm Im}\rho_n$ is a subgroup of the finite cyclic group $\mu_{2n+1}$, we have ${\rm Im}\rho_n\simeq \C_{\#{\rm Im}\rho_n}$. It remains to show that $\#{\rm Ker}\rho_n$ and $\#{\rm Im}\rho_n$ are coprime. 
By Theorem \ref{thm:1} (1), (2), there exists a non-negative integer $\ell$ such that 
\[
\#{\rm Aut}(C)=\#{\rm Ker}\tau_n\cdot \#{\rm Im}\tau_n 
=\#\C_2\cdot \#U_n\cdot \#\C_{\#{\rm Im}\rho_n} 
=2\cdot 2^{\ell}\cdot \#{\rm Im}\rho_n
\]
holds. Therefore we have
$
\#{\rm Ker}\rho_n=\#{\rm Aut}(C)/\#{\rm Im}\rho_n
=2^{\ell+1}$,
but on the other hand, $\#{\rm Im}\rho_n$ is an odd number because it is a divisor of $2n+1$. 
Hence $\#{\rm Ker}\rho_n$ and $\#{\rm Im}\rho_n$ are coprime. 
\end{enumerate}
\end{proof}

\subsection{Proof of Theorem \ref{thm:2}}
In this subsection, we prove Theorem \ref{thm:2}, namely, for $n\leq 6$ we determine the group structure of the reduced automorphism group ${\rm RA}(C)$ of
\[
C:\quad y^2-y =x(x^{2n}+a_{n-1}x^{2(n-1)}+\cdots +a_1x^2+a_0)
\]
for $a_i\in K$.
For the proof, we used the computational software Magma (cf. \cite{MAGMA}). 
\begin{proof}
From the expression for the coefficients of the curve used in the proof of Theorem \ref{thm:1} (3), $c(\ell)$, it can be seen that once $\alpha$ and $\beta$ are determined, $\gamma_n, \ldots, \gamma_1$ are determined uniquely, and $\gamma_0$ is determined in two ways. From this fact and Theorem \ref{thm:1} (2),
in order to determine the group structure of the reduced automorphism group ${\rm RA}(C)$, it is sufficient to find the number of values that $\alpha$ and $\beta$ can take. For each $n=1,\ldots, 6$, we derived the number of possible values of $\alpha$ and of $\beta$ (that is, the order of ${\rm Im}\rho_n$ and of $U_n$), as well as their conditional formulas, using the computational software Magma (cf. \cite{MAGMA}). 
More precisely, we used Magma to find the Gr\"obner basis of the ideal generated by polynomials that are necessary and sufficient for the coordinate transformation of the curve $C$ to be an automorphism, and analyzed the obtained output. Considering the fiber structure ${\rm Im} \tau_n \to {\rm Im} \rho_n$: $(\alpha,\beta) \mapsto \alpha$, we used a lexicographic order satisfying $\beta \succ \alpha$ as the monomial order to compute
the Gr\"obner basis above. 

\begin{algorithm}[htbp]
\begin{algorithmic}[1]
\renewcommand{\algorithmicrequire}{\textbf{Input:}}
\renewcommand{\algorithmicensure}{\textbf{Output:}}
\algnewcommand{\IIf}[1]{\State\algorithmicif\ #1\ \algorithmicthen}
\algnewcommand{\EndIIf}{\unskip\ \algorithmicend\ \algorithmicif}
\caption{\ Computing Gr\"obner basis}\label{Cariter-Manin}
\label{carier-manin:algrm}
\Require A natural number $n\in \mathbb{N}$
\Ensure The Gr\"obner basis $G$ of the ideal generated by polynomials that are necessary and sufficient for the coordinate transformation of the curve to be an automorphism
\State Set the polynomial $F:=y^2-y-x(x^{2n}+\sum_{i=0}^{n-1}a_{i}x^{2i})$ and the coordinate transformation $\varphi:=(x\mapsto \alpha x+\beta,\ y\mapsto y+\sum_{i=0}^n \gamma_ix^i)$
\State Compute $H:=F-\varphi(F)$
\State Set $L:=\emptyset$ and the set $M$ of monomials contained in the polynomial $H$
\For{the monomial $m$ in $M$}
\State $L:=L\cup \{{\rm the\ coefficient\ of\ }m\}$
\EndFor
\State Set the ideal $I$ generated by $L$
\State Compute the Gr\"obner basis $G$ of $I$ with respect to the lexicographic order
$\gamma_n \succ \cdots \succ \gamma_0 \succ \beta \succ \alpha \succ a_0 \succ \cdots \succ a_{n-1}$\\
\Return $G$
\end{algorithmic}
\end{algorithm}
If we put $\alpha = 1$, i.e., we replace ``$L:=\emptyset$" by ``$L:=\{\alpha -1 \}$" in the third line of Algorithm \ref{carier-manin:algrm}, the algorithm provides a Gr\"obner basis determining $U_n$.

Here, as an example, the steps of the proof are described concretely for $n=3$. For the curve 
\[
C:y^2-y=x(x^6+a_2x^4+a_1x^2+a_0)
\]
over $K$ and its coordinate transformation 
\[
\left\{
\begin{array}{l}
x \mapsto \alpha x+\beta, \\
y \mapsto y+\gamma_3x^3+\gamma_2x^2+\gamma_1x+\gamma_0, 
\end{array}
\right.
\]
the necessary and sufficient condition for this coordinate transformation to be an automorphism is that the following holds. 
\[
\left\{
\begin{array}{l}
\alpha^7=1\\
\alpha^6\beta+\gamma_3^2=0\\
\alpha^5\beta^2+a_2\alpha^5=a_2\\
\alpha^4\beta^3+a_2\alpha^4\beta+\gamma_2^2=0\\
\alpha^3\beta^4+a_1\alpha^3+\gamma_3=a_1\\
\alpha^2\beta^5+a_1\alpha^2\beta+\gamma_1^2+\gamma_2=0\\
\alpha\beta^6+a_2\alpha\beta^4+a_1\alpha\beta^2+a_0\alpha+\gamma_1=a_0\\
\beta^7+a_2\beta^5+a_1\beta^3+a_0\beta+\gamma_0^2+\gamma_0=0
\end{array}
\right.
\]
By computing the Gr\"obner basis of the ideal generated by this polynomials using Magma, we obtain the following output. 
\[
\left\{
\begin{array}{l}
    \gamma_3 + \alpha^3a_1 + \alpha^3a_2^2 + a_1 + a_2^2 =0\\
    \gamma_2 + \alpha^3a_1^3 + \alpha^3a_1^2a_2^2 + \alpha^3a_1a_2^4 + 
        \alpha^3a_2^6 \\
        \quad + \alpha^2a_0^2 + \alpha^2a_1^3 + \alpha^2a_1a_2^4 + 
        \alpha^2a_2^6 + a_0^2 + a_1^2a_2^2 =0\\
    \gamma_1 + \alpha^3a_1a_2 + \alpha^3a_2^3 + \alpha a_0 + \alpha a_1a_2 + a_0 + 
        a_2^3 =0\\
    \gamma_0^2 + \gamma_0 + \alpha^6a_1^2a_2^3 + \alpha^6a_2^7 + \alpha^3a_1^3a_2 + 
        \alpha^3a_1^2a_2^3 + \alpha^3a_1a_2^5 + \alpha^3a_2^7 + 
        \alpha^2a_1^3a_2 \\
        \quad +\alpha^2a_1^2a_2^3 + \alpha^2a_1a_2^5 + 
        \alpha^2a_2^7 + \alpha a_0a_1^2 + \alpha a_0a_2^4 + \alpha a_1^3a_2 + 
        \alpha a_1a_2^5 + a_0a_1^2 \\
        \qquad + a_0a_2^4 + a_1^3a_2 + a_1^2a_2^3 + 
        a_1a_2^5 + a_2^7 =0\\
    \beta + \alpha a_1^2 + \alpha a_2^4 + a_1^2 + a_2^4 =0\\
    \alpha^7 + 1 =0\\
    \alpha a_0^4 + \alpha a_1^2a_2 + \alpha a_2^{12} + a_0^4 + a_1^2a_2 + a_2^{12} =0\\
    \alpha a_1^4 + \alpha a_2^8 + \alpha a_2 + a_1^4 + a_2^8 + a_2 =0
\end{array}
\right.
\]
We have $\alpha^7=1$ from the $6$th eqation. Also, from $7$th and $8$th equation, if both $a_2^8+a_2+a_1^4=0$ and $a_2^{12}+a_2a_1^2+a_0^4=0$ hold, then $\alpha^7=1$; otherwise, $\alpha=1$ is obtained. This allows us to determine the number of values $\alpha$ can take. Next, assuming $\alpha=1$, we find the number of values $\beta$ can take. We have $\beta=0$ by substituting $\alpha=1$ for the $5$th equation, so the number of values that $\beta$ can take is $1$. In this way, we analyzed the output (Gr\"obner basis) obtained with Magma and proved Theorem \ref{thm:2} by determining the order of ${\rm Im}\rho_n$ and of $U_n$. The source code used is available on \cite{Yamaguchi}. For each $n=1,\ldots,6$, the details are described below. 
\begin{enumerate}
\item[\em $(1)$] \underline{The case of $n=1$:} By substituting $\ell=3$ for $c(\ell)$, we obtain $\alpha^3=1$. Since the formula concerning $\alpha$ is only this formula, $\#{\rm Im}\rho_1=3$ always holds. Next, let $\alpha=1$, and show that $U_n\simeq \C_2^2$, that is, $\beta$ takes $4$ different values. From the calculation results of Magma
, it follows that $\beta^4+\beta=0$, and since this equation does not have a multiple root, it can be seen that $\beta$ takes on $4$ different values. 
\item[\em $(2)$] \underline{The case of $n=2$:}
By substituting $\ell=5$ for $c(\ell)$, we obtain $\alpha^5=1$. Also, we obtain $a_1(\alpha+1)=0$ from the calculation results of Magma. 
Since ${\rm Im}\rho_2$ is a cyclic group of order $1$ or $5$, when $a_1\ne0$ holds, we have $\#{\rm Im}\rho_2=1$, and when $a_1=0$ holds, we have $\#{\rm Im}\rho_2=5$. Next, let $\alpha=1$, and show that $U_n\simeq \C_2^4$, that is, $\beta$ takes $16$ different values. From the calculation results of Magma, 
it follows that 
\[
\beta^{16}+a_1^4\beta^8+a_1^2\beta^2+\beta=0
\]
and since this equation does not have a multiple root, it can be seen that $\beta$ takes on $16$ different values. 
\item[\em $(3)$] \underline{The case of $n=3$:}
By substituting $\ell=7$ for $c(\ell)$, we obtain $\alpha^7=1$. Also, we obtain 
\[
\left\{
\begin{array}{l}
(a_2^8+a_2+a_1^4)(\alpha+1)=0 \\
(a_2^{12}+a_2a_1^2+a_0^4)(\alpha+1)=0
\end{array}
\right.
\]
from the calculation results of Magma. 
Since ${\rm Im}\rho_3$ is a cyclic group of order $1$ or $7$, when $a_2^8+a_2+a_1^4=a_2^{12}+a_2a_1^2+a_0^4=0$ holds, we have $\#{\rm Im}\rho_3=7$, and otherwise, we have $\#{\rm Im}\rho_3=1$. Next, we show that $U_n=\{e\}$. From the proof of Theorem \ref{thm:1} (3), when $n$ is an odd number greater than or equal to $3$, if $\alpha=1$, then $\beta=0$. Therefore $U_n=\{e\}$ holds. 
\item[\em $(4)$] \underline{The case of $n=4$:}
By substituting $\ell=9$ for $c(\ell)$, we obtain $\alpha^9=1$. Also, we obtain 
\[
\left\{
\begin{array}{l}
a_3(\alpha+1)=0 \\
a_2(\alpha+1)=0 \\
a_1^2(\alpha^3+1)=0
\end{array}
\right.
\]
from the calculation results of Magma. 
Since ${\rm Im}\rho_4$ is a cyclic group of order $1, 3$, or $9$, when either $a_3\ne0$ or $a_2\ne0$ holds, we have $\#{\rm Im}\rho_4=1$, when $a_3=a_2=0$ and $a_1\ne0$ holds, we have $\#{\rm Im}\rho_4=3$, and when $a_3=a_2=a_1=0$ holds, we have $\#{\rm Im}\rho_4=9$. Next, let $\alpha=1$, and show that $U_n=\{e\}$ when $a_3\ne0$ holds, and that $U_n\simeq \C_2^6$ when $a_3=0$ holds. 
\begin{enumerate}
    \item[\em $({\rm i})$] \underline{The case of $a_3\ne0$:}
    From the calculation results of Magma, 
    we obtain $a_3\beta=0$. Since $a_3\ne0$, $\beta=0$ holds, so we have $U_n=\{e\}$. 
    \item[\em $({\rm ii})$] \underline{The case of $a_3=0$:}
    From the calculation results of Magma, 
    it follows that
    \[
    \beta^{64}+a_2^8\beta^{32}+a_1^8\beta^{16}+a_1^4\beta^4+a_2^2\beta^2+\beta=0
    \]
    and since this equation does not have a multiple root, it can be seen that $\beta$ takes on $64$ different values. Therefore $U_n\simeq \C_2^6$. 
\end{enumerate}
\item[\em $(5)$] \underline{The case of $n=5$:}
By substituting $\ell=11$ for $c(\ell)$, we obtain $\alpha^{11}=1$. Also, we obtain 
\[
\left\{
\begin{array}{l}
a_3(\alpha+1)=0 \\
(a_4+a_2^4)(\alpha+1)=0 \\
(a_4^8+a_1^2)(\alpha+1)=0 \\
(a_4^{40}+a_4^{18}+a_4a_2^2+a_0^8)(\alpha+1)=0
\end{array}
\right.
\]
from the calculation results of Magma. 
Since ${\rm Im}\rho_5$ is a cyclic group of order $1$ or $11$, when $a_3=a_4+a_2^4=a_4^8+a_1^2=a_4^{40}+a_4^{18}+a_4a_2^2+a_0^8=0$ holds, we have $\#{\rm Im}\rho_5=11$, and otherwise, we have $\#{\rm Im}\rho_5=1$. Next, we show that $U_n=\{e\}$. From the proof of Theorem \ref{thm:1} (3), when $n$ is an odd number greater than or equal to $3$, if $\alpha=1$, then $\beta=0$. Therefore $U_n=\{e\}$ holds. 
\item[\em $(6)$] \underline{The case of $n=6$:}
By substituting $\ell=13$ for $c(\ell)$, we obtain $\alpha^{13}=1$. Also, we obtain $a_5(\alpha+1)=0, a_3(\alpha+1)=0$ from the calculation result of Magma, 
 where we used the grevlex order to terminate the computation. 
Since ${\rm Im}\rho_6$ is a cyclic group of order $1$ or $13$, when $a_5=a_3=0$ does not hold, we have $\#{\rm Im}\rho_6=1$. By performing the calculation with $a_5=a_3=0$, we obtain the following three equations. 
\[
\left\{
\begin{array}{l}
(a_4^4+a_2^2)(\alpha+1)=0, \\
(a_4+a_1^{16})(\alpha+1)=0, \\
(a_4^{24}+a_4^4a_1^8+a_4a_1^4+a_0^8)(\alpha+1)=0. 
\end{array}
\right.
\]
From the considerations so far, it can be seen that the necessary and sufficient condition for $\#{\rm Im}\rho_6=13$ is $a_5=a_3=a_4^4+a_2^2=a_4+a_1^{16}=a_4^{24}+a_4^4a_1^8+a_4a_1^4+a_0^8=0$. Next, we consider the group structure of $U_n$. From the calculation results of Magma with $\alpha=1$, we obtain the following three equations. 
\[
\left\{
\begin{array}{l}
(a_5a_3^4+a_5^{13}+1)\beta=0, \\
t\beta=0, \\
\beta^2+a_5(a_3^2+a_5^6)\beta=0,
\end{array}
\right.
\]
where $t$ is as in \eqref{eq:20260412}.
When the necessary condition $a_5=0$ for $\#{\rm Im}\rho_6=13$ is satisfied, from the third equation, $\beta=0$, that is, $U_n=\{e\}$ holds. Therefore, if $\#{\rm Im}\rho_6=13$ holds, the group structure of ${\rm RA}(C)$ is determined as $\C_{13}$. On the other hand, when $\#{\rm Im}\rho_6=1$ holds (that is, when at least one of $a_5, a_3, a_4^4+a_2^2, a_4+a_1^{16}, a_4^{24}+a_4^4a_1^8+a_4a_1^4+a_0^8$ is nonzero), from the three equations above, if $a_5(a_3^2+a_5^6)\ne0$ and $a_5a_3^4+a_5^{13}+1=t=0$ hold, then $\beta$ takes two possible values, so we have $U_n\simeq \C_2$, otherwise $\beta=0$, so we have $U_n=\{e\}$. 
\end{enumerate}
Regarding the order of the automorphism group ${\rm Aut}(C)$, since $\#{\rm Aut}(C)=2\cdot\#{\rm RA}(C)=2\cdot\#U_n\cdot \#{\rm Im}\rho_n$ holds, the order of ${\rm Aut}(C)$ can be immediately determined from the group structure of ${\rm RA}(C)$. 
\end{proof}

\subsection{Proof of Theorem \ref{thm:3}}
In this subsection, we prove Theorem \ref{thm:3}. That is, we determine ${\rm RA}(C_g^{SZ})$ of Artin-Schreier curves $C_g^{SZ}$ of Scholten-Zhu type.
\begin{proof}
In the case of $g=1, \ldots, 6$, it immediately follows from Theorem \ref{thm:2}. Using a method similar to the proof of Theorem \ref{thm:2}, we show the case $g=8$ and $9$. 
\begin{enumerate}
\item[\em $(1)$] \underline{The case of $g=8$:}
By substituting $\ell=17$ for $c(\ell)$, we obtain $\alpha^{17}=1$. Also, we obtain 
\[
\left\{
\begin{array}{l}
c_9(\alpha+1)=0 \\
c_5^2(\alpha+1)=0 \\
c_3^4(\alpha+1)=0
\end{array}
\right.
\]
from the calculation results of Magma. 
Since ${\rm Im}\rho_8$ is a cyclic group of order $1$ or $17$, when $c_9=c_5=c_3=0$ holds, we have $\#{\rm Im}\rho_8=17$, and otherwise, we have $\#{\rm Im}\rho_8=1$. Next, let $\alpha=1$, and show that $U_n\simeq \C_2^8$, that is, $\beta$ takes $256$ different values. From the calculation results of Magma, it follows that 
\[
\beta^{256}+c_9^{16}\beta^{128}+c_5^{16}\beta^{64}+c_3^{16}\beta^{32}+c_3^8\beta^8+c_5^4\beta^4+ c_9^2\beta^2+\beta=0
\]
and since this equation does not have a multiple root, it can be seen that $\beta$ takes on $256$ different values. 
\item[\em $(2)$] \underline{The case of $g=9$:}
By substituting $\ell=19$ for $c(\ell)$, we obtain $\alpha^{19}=1$. Also, we obtain $c^{32}(\alpha+1)=0$ from the calculation results of Magma. 
Since ${\rm Im}\rho_9$ is a cyclic group of order $1$ or $19$, when $c\ne0$ holds, we have $\#{\rm Im}\rho_9=1$, and when $c=0$ holds, we have $\#{\rm Im}\rho_9=19$. Next, we show that $U_n=\{e\}$. From the proof of Theorem \ref{thm:1} (3), when $n$ is an odd number greater than or equal to $3$, if $\alpha=1$, then $\beta=0$. Therefore $U_n=\{e\}$ holds. 
\end{enumerate}
\end{proof}

\section{Concluding remarks}
In this paper, we considered the nonsingular curve over an algebraically closed field $K$ of characteristic $2$ defined by
\[
    C:y^2-y
    =x(x^{2n}+a_{n-1}x^{2(n-1)}+\cdots +a_1x^2+a_0)
\]
for $a_i \in K$. These are hyperelliptic curves of $2$-rank $0$. We examined the group sturcture and the order of the automorphism group ${\rm Aut}(C)$, as well as the reduced automorphism group ${\rm RA}(C)$. We were able to determine
the group structure of ${\rm RA}(C)$ when $n\leq 6$, but for bigger $n$
we were unable to determine it 
due to computational difficulties. Of course, by improving the algorithm, enhancing the computational environment, and investing a substantial amount of time, it may be possible to obtain similar results for slightly larger values of $n$.
Based on these experiments, we formulated analogues of Oort's conjecture for these families of curves, see Conjectures 1 and 2 in Introduction.

As a future work, we would like to have a theoretical approach toward general results on
the group structure of ${\rm Aut}(C)$ 
as well as the detailed group structure of ${\rm RA}(C)$ for any $n$. 
The remaining problem for odd $n$ is to classify the image of $\rho_n$. 
When $n$ is even, in addition to the above issues, determining the structure of the kernel also becomes a major problem. However
when $n$ is a power of $2$, say $n=2^m$, as far as our experiments indicate, only the following two types of groups appeared as $U_n$. 
\[
U_n\simeq \begin{cases}
        \C_1 &  \text{for specific } (a_i) \text{ chosen in our experiments,} \\
        \C_2^{2(m+1)} & \text{if } a_i=0 \text{ unless }i \text{ is a power of } 2,
        \end{cases}
\]
where we note ${\rm RA}(C) \simeq U_n\rtimes \C_{\#{\rm Im}\rho_n}$
by Theorem \ref{thm:1}.
However, it is uncertain whether one can assert that only these two cases occur, and we also regard this as an interesting problem, much like Conjectures 1 and 2.


\begin{thebibliography}{99}
\bibitem{MAGMA}W.\ Bosma, J.\ Cannon and C.\ Playoust: \textit{The MAGMA algebra system. I. The user language}, J.\ Symb.\ Comput.\,{\bf 24}, No.\,3-4, 235--265, 1997.

\bibitem{Van Der Geer:rmc}G.\ van der Geer and M.\ van der Vlugt: {\it Reed-Muller codes and supersingular curves. I}, Compos. Math. {\bf 84}, No.\,3, 333--367, 1992.

\bibitem{Van Der Geer:esc}G.\ van der Geer and M.\ van der Vlugt: {\it On the existence of supersingular curves of given genus},  J.\ Reine Angew. Math. {\bf 458}, 53--62, 1994.

\bibitem{Hartshorne}R.\ Hartshorne: {\it Algebraic Geometry},  Graduate Texts in Mathematics {\bf 52}, Springer, 1997.

\bibitem{Ibukiyama}T.\ Ibukiyama: {\it Principal polarizations of supersingular abelian surfaces}, J.\ Math.\ Soc.\ Japan {\bf 72}, No.\,4, 1161--1180, 2020.

\bibitem{KYY}
V.\ Karemaker, F.\ Yobuko and C.-F.\ Yu:
{\it Mass formula and Oort's conjecture for supersingular abelian threefolds},
Adv. Math. {\bf 386}, Article ID 107812, 
(2021).

\bibitem{KY}
V.\ Karemaker and C.-F.\ Yu:
{\it Supersingular Ekedahl-Oort strata and Oort's conjecture}, 
arXiv:2406.19748 [math.NT]

\bibitem{IKO}
T.\ Ibukiyama, T.\ Katsura and F.\ Oort:
{\it Supersingular curves of genus two and class numbers},
Compos. Math. {\bf 57}, 127--152 (1986).

\bibitem{Igusa}
J.\ Igusa:
{\it Arithmetic variety of moduli for genus two}, 
Ann.\ Math.\ (2) {\bf 72}, 612--649 (1960).

\bibitem{LR}
R.\ Lercier and C.\ Ritzenthaler:
{\it Hyperelliptic curves and their invariants: Geometric, arithmetic and algorithmic aspects},
Journal of Algebra,
Volume {\bf 372},
2012, 
595--636,


\bibitem{Pries} R.\ Pries and H.\ J.\ Zhu: {\it The $p$-rank stratification of Artin-Schreier curves}, Ann.\ Inst.\ Fourier {\bf 62}, No.\,2, 707--726, 2012.

\bibitem{Rotman}J.\ J.\ Rotman: {\it An Introduction to the Theory of Groups}, Graduate Texts in Mathematics {\bf 148}, Springer, 1995.

\bibitem{SZ:hyperell}
J.\ Scholten and H.\ J.\ Zhu:
{\it Hyperelliptic curves in characteristic $2$},
Int.\ Math.\ Res.\ Not.\ 2002, No.\,17, 905--917 (2002).

\bibitem{Scholten} J.\ Scholten and H.\ J.\ Zhu: {\it Families of supersingular curves in characteristic $2$}, Math.\ Res.\ Lett.\ {\bf 9}, No.\,5-6, 639-650, 2002.

\bibitem{Viehmann} 
E.\ Viehmann:
{\it Oort's conjecture on automorphisms of generic supersingular abelian varieties},
arXiv:2603.06033 [math.AG]




\bibitem{Yamaguchi}\texttt{https://github.com/Kohtaro-Yamaguchi/code_Magma.git}
\end{thebibliography}
\end{document}